\documentclass[a4paper,reqno,10pt]{amsart}

\usepackage[english]{babel}
\usepackage[T1]{fontenc}
\usepackage[utf8]{inputenc}

\usepackage{amsmath}
\usepackage{amssymb}
\usepackage{amsthm}
\usepackage{xcolor}
\usepackage{geometry}
\usepackage{hyperref}

\newtheorem{lemma}{Lemma}[section]
\newtheorem{definition}{Definition}[section]
\newtheorem{theorem}{Theorem}[section]
\theoremstyle{remark}
\newtheorem{remark}[lemma]{Remark}

\usepackage[capitalise]{cleveref}

\newcommand{\T}{\mathbb{T}^2}
\newcommand{\dive}{\mathop{\mathrm {div}}}
\numberwithin{equation}{section}


\begin{document}

\title[Uniqueness 2D compressible Navier-Stokes equations]{Uniqueness of Solutions to the 2D Compressible Navier-Stokes Equations for Bounded Densities}

\author[Lorenzo Pescatore and Alessandro Violini]{Lorenzo Pescatore and Alessandro Violini}
\address[L. Pescatore]{DISIM - Dipartimento di Ingegneria e Scienze dell'Informazione e Matematica\\ Universit\`a  degli Studi dell'Aquila \\Via Vetoio \\ 67100 L'Aquila \\ Italy}
\email[]{\href{lorenzo.pescatore@}{lorenzo.pescatore@univaq.it}}
\address[A. Violini]{Universit\"{a}t Basel, Departement Mathematik und Informatik, Spiegelgasse 1, 4051 Basel, Switzerland}
\email[]{\href{alessandro.violini@}{alessandro.violini@unibas.ch}}
\begin{abstract}
We prove the uniqueness of global solutions to the two dimensional barotropic compressible Navier--Stokes equations within the regularity class introduced by Danchin and Mucha \cite{DM2023}, for general $C^1$ pressure laws. A crucial assumption is that both solutions have bounded divergence, locally integrable in time, which ensures that their densities remain bounded without excluding vacuum or requiring spatial regularity of the density.
\end{abstract}
\maketitle
\section{Introduction}
We consider the barotropic compressible Navier--Stokes equations on the two dimensional torus $\T$,
\begin{equation}\label{main system}
\begin{cases}
\partial_t\rho+\dive(\rho u)=0,\\
\partial_t(\rho u)+\dive(\rho u\otimes u)+\nabla P(\rho)= \mu\Delta u+(\lambda+\mu)\nabla\dive u,
\end{cases}
\end{equation}
where the shear and bulk viscosity coefficients satisfy $\mu>0 $ and $ \nu:=\lambda+2\mu>0.$ Throughout the paper, we consider pressure laws satisfying
\begin{equation}\label{pressure}
P\in C^1([0,\infty)),
\qquad
P(0)=0.
\end{equation}
We emphasize that neither monotonicity nor convexity is imposed on $P.$ This class includes, in particular, the standard isentropic pressure laws
\begin{equation}\label{press Danchin-Mucha}
P(\rho)=a\rho^\gamma,
\qquad a>0,
\qquad \gamma\ge1.
\end{equation}

Uniqueness of global solutions is a central question in the analysis of nonlinear partial differential equations and remains one of the most challenging problems in the mathematical theory of fluid mechanics. For the incompressible Navier--Stokes equations, the pioneering work of Ladyzhenskaya \cite{Ladyzhenskaya1959} established uniqueness of Leray--Hopf weak solutions \cite{Leray1934,Hopf1951} in two space dimensions.
In three dimensions, the Prodi--Serrin criteria \cite{Prodi1959,Serrin1962} yield conditional uniqueness under additional space time integrability assumptions on the velocity.
Nonuniqueness within the three dimensional Leray--Hopf class was established for suitable forcing by Albritton, Bru\`e and Colombo \cite{Albritton-Bruè-Colombo2022}, \cite{Albritton-Bruè-Colombo2023}, by exploiting the instability of self similar profiles. Recently, Hou, Wang and Yang \cite{Hou-Wang-Yang2026} have presented a computer assisted proof of nonuniqueness for the unforced system. Outside the Leray--Hopf class, convex integration yields nonuniqueness of weak solutions with finite kinetic energy \cite{BuckmasterVicol2019}, with further sharp nonuniqueness results obtained by Cheskidov and Luo \cite{CheskidovLuo2022}.

The scenario is considerably more involved for compressible fluid equations, where both existence and uniqueness face additional difficulties even in two space dimensions. The hyperbolic character of the continuity equation allows density irregularities and discontinuities to propagate, while the pressure inherits the limited spatial regularity of the density. The initial upper and positive lower bounds of the density are not automatically preserved. Moreover, the presence of vacuum prevents, in general, direct control of the velocity and naturally leads instead to estimates for density weighted quantities. These features prevent a direct extension of the incompressible theory, for both existence and uniqueness.
\medskip

The global existence theory for finite energy weak solutions to the barotropic compressible Navier--Stokes equations was initiated by Lions \cite{Lions1998} for adiabatic exponents $\gamma\ge \frac{3d}{d+2}$ in two and three space dimensions. The a priori bound obtained by Lions involves the energy
\begin{equation*}
E(t):=\int_{\T}\left(\frac12\rho|u|^2+h(\rho)\right)\,dx,
\qquad
h(\rho):=\rho\int_1^\rho\frac{P(s)}{s^2}\,ds,
\end{equation*}
which satisfies
\begin{equation}\label{energy ineq}
E(t)+\int_0^t
\left(
\mu\|\nabla\mathbb P u(s)\|_{L^2}^2
+\nu\|\dive u(s)\|_{L^2}^2
\right)\,ds
\leq E(0).
\end{equation}
Here $\mathbb P$ denotes the Leray projector, and $h$ is the internal energy per unit volume, satisfying $\rho h''(\rho)=P'(\rho)$ for $\rho>0$. Lions' existence result was  subsequently improved by Feireisl, Novotn\'y and Petzeltov\'a \cite{FNP2001}, who lowered the restriction to $\gamma>\frac d2,$ including in particular the physically relevant monoatomic exponent $\gamma=5/3$ in three dimension. A fundamental obstruction arising in these analyses is the difficulty of establishing strong compactness of the density, needed to identify the weak limit of the nonlinear pressure. To overcome this issue, the theory developed by Feireisl and collaborators makes use of  suitable truncations of the density, together with the renormalized formulation of the continuity equation. These features are further enhanced by the analysis of the effective viscous flux and the introduction of the oscillation defect measure. Beyond the classical Lions-Feireisl framework, Bresch and Jabin \cite{BJ2018} developed a quantitative compactness method covering certain non monotone pressure laws and anisotropic viscous stresses. For density dependent viscosities, which may degenerate near vacuum, the Bresch--Desjardins algebraic relation
between the shear and bulk viscosity coefficients yields additional entropy estimates and density regularity \cite{BreschDesjardins2002,BreschDesjardins2004}. Building on this structure, Bresch, Vasseur and Yu \cite{BVY2022} established global entropy-weak solutions for a broad class of such viscosities and power law pressures with any $\gamma>1$.
\medskip

Our analysis focuses on solutions whose density remains bounded on every finite time interval. We consider initial data satisfying
\begin{equation}\label{eq:initial-data}
\rho_0\in L^\infty(\T),
\qquad
\rho_0\ge0,
\qquad
u_0\in L^2(\T),
\end{equation}
thus allowing for discontinuous initial densities and vacuum regions. For simplicity, we normalize the initial mass to one, so that conservation of mass gives
\begin{equation}\label{eq:unitmass}
\int_{\T}\rho(t,x)\,dx
=\int_{\T}\rho_0(x)\,dx=1,
\qquad t\ge0.
\end{equation}

Propagating the initial $L^\infty$ bound on the density requires control beyond the basic energy estimates. To this end, a key hypothesis in our framework is $\dive u\in L^1_{\rm loc}([0,\infty);L^\infty(\T))$, which, together with the Sobolev regularity of the velocity specified below, allows us to propagate this bound through the continuity equation. This motivates the following class of solutions. 
\begin{definition}\label{def:bounded-divergence-solution}
We say that $(\rho,u)$ is a solution of \eqref{main system} with initial data \eqref{eq:initial-data} if $\rho\ge0$, $(\rho,u)$ solves \eqref{main system} in the sense of distributions and satisfies
\begin{equation*}
\rho\in L^\infty_{\rm loc}([0,\infty);L^\infty(\T)),
\qquad
\rho\in C([0,\infty);L^p(\T))\quad\forall p<\infty,
\qquad
\end{equation*}
\begin{equation*}
\sqrt{\rho}\,u\in C([0,\infty);L^2(\T)), \qquad
\nabla u\in L^2_{\rm loc}([0,\infty);L^2(\T)),
\qquad
\dive u\in L^1_{\rm loc}([0,\infty);L^\infty(\T)).
\end{equation*}
In addition, we require \eqref{energy ineq} to hold with equality
for every $t\ge0$.
\end{definition}

This class is natural from the viewpoint of renormalized transport theory: the Sobolev regularity of the velocity and the integrable bound on its divergence are precisely the velocity assumptions of the DiPerna--Lions framework \cite{DiPernaLions1989} (see also \cite{Ambrosio2004}). We use the corresponding renormalized density evolution, which also allows us to transport nonlinear functions of the density. In particular, the pressure satisfies
\[
\partial_tP(\rho)+\dive\bigl(P(\rho)u\bigr)
=\bigl(P(\rho)-\rho P'(\rho)\bigr)\dive u.
\]

The divergence bound also controls compression and propagates the initial upper bound on $\rho$. Indeed, if $X$ is the associated regular Lagrangian flow, its Jacobian factor is given by
\begin{equation}\label{eq:jacobian-flow}
J(t,\xi):=\exp\left(\int_0^t\dive u(s,X(s,\xi))\,ds\right).
\end{equation}
This factor describes the change in volume along the flow: compression increases the density, while expansion decreases it. The integrable bound on the divergence ensures that $J$ remains bounded above and away from zero on every finite time interval. Conservation of mass along the flow gives
\begin{equation}\label{eq:lagrangian-density}
\rho(t,X(t,\xi))J(t,\xi)=\rho_0(\xi).
\end{equation}
Consequently, the initial upper bound on the density propagates according to
\begin{equation}\label{eq:density-Linfty-bound}
\|\rho(t)\|_{L^\infty}
\leq\|\rho_0\|_{L^\infty}
\exp\left(\int_0^t\|\dive u(s)\|_{L^\infty}\,ds\right).
\end{equation}
Thus the density and pressure remain bounded on every finite time interval, and the effective viscous flux satisfies
\begin{equation}\label{eq:effective-flux}
G:=\nu\dive u-P(\rho)
\in L^1_{\rm loc}([0,\infty);L^\infty(\T)).
\end{equation}
These transport properties provide the starting point for our analysis.

Existence of solutions of the type defined in \cref{def:bounded-divergence-solution} was proved by Danchin and Mucha in \cite{DM2023} under additional Sobolev regularity of the initial velocity $u_0$.

\begin{theorem}[\cite{DM2023}]\label{thm:DM}
Assume that the pressure $P$ satisfies \eqref{pressure} and, in addition, $\rho \to \rho^{-1} P(\rho)$ is non decreasing and non negative. Let $(\rho_0,u_0)$ satisfy \eqref{eq:initial-data}. Fix some $K>0$ and assume in addition that
\begin{equation}\label{eq:DM-div-smallness}
u_0 \in H^1(\T), \qquad \|\dive u_0\|_{L^2}
\le
K\nu^{-1/2}.
\end{equation}
Then, there exists a positive number $\nu_0>0$, depending only on the data, the pressure law, $K$ and $\mu$, such that, if $\nu\ge\nu_0$, system \eqref{main system} admits a global in time solution (in the sense of \cref{def:bounded-divergence-solution}) satisfying in addition the following regularity estimates
\begin{equation}\label{eq:stronger-u2-regularity}
u \in L^\infty_{\rm loc}([0,\infty);H^1(\T)), \qquad 
\sqrt{\rho}\,\dot u, \; 
\sqrt t\,\nabla\dot u
\in L^2_{\rm loc}([0,\infty);L^2(\T)).
\end{equation}
\end{theorem}

We also mention the recent result \cite{LeiXiong2025}, which extends \cref{thm:DM} by removing the quantitative bound on the initial divergence in \eqref{eq:DM-div-smallness}, while still requiring sufficiently large bulk viscosity.
\\

We refer to solutions satisfying \eqref{eq:stronger-u2-regularity} as \emph{Danchin--Mucha solutions}. Danchin and Mucha proved uniqueness within their solution class for linear pressure laws and conjectured its extension to convex and strictly increasing pressure laws satisfying \eqref{pressure}. Our aim is to establish a uniqueness result for general pressure laws satisfying \eqref{pressure}, where one solution is in the larger class of \cref{def:bounded-divergence-solution} and the other one is a Danchin--Mucha solution.

The different regularity assumptions on the two solutions suggest a comparison with the classical weak--strong uniqueness framework.
For the compressible Navier--Stokes equations, such results ensure that a weak solution coincides with a sufficiently regular strong solution with the same initial data for as long as the latter exists; see Germain \cite{Germain2011}, Feireisl, Jin and Novotn\'y \cite{FJN2012}, and Feireisl \cite{Feireisl2019} for non monotone pressure laws. These results rely on relative energy estimates and require regularity of the strong solution's density that excludes general bounded, discontinuous densities. They therefore do not cover our setting, even when both solutions are Danchin--Mucha solutions.

Uniqueness for compressible flows with discontinuous density has also been studied in the intermediate regularity classes introduced by Hoff \cite{Hoff1995,Hoff2006}, between the finite energy weak framework \cite{Lions1998,FNP2001} and the classical strong solution framework \cite{Nash1962}. For linear pressure laws, uniqueness is obtained through Lagrangian arguments and estimates in negative Sobolev spaces. For general pressure laws, Danchin, Fanelli and Paicu \cite{DFP2020} established local well posedness for densities close to a positive constant under a striated regularity assumption. This assumption controls derivatives along suitable nondegenerate families of directions while allowing discontinuities, and yields an integrable Lipschitz bound on the velocity. Liao and Zodji \cite{LZ2025} subsequently obtained global well posedness in the whole space for sufficiently large bulk viscosity and pressure laws $P\in C^2$, allowing vacuum while retaining striated regularity of the density. For density dependent viscosities, Zodji \cite{Zodji2026} proved global existence and uniqueness in two dimensions for densities bounded away from vacuum and piecewise H\"older continuous across a regular interface, under suitable smallness assumptions.
\medskip

Our result requires no directional regularity of the density and establishes uniqueness in the class of \cref{def:bounded-divergence-solution}, provided that one solution satisfies \eqref{eq:stronger-u2-regularity}. In particular, the other solution need only satisfy the bounds in that definition, including $L^1$ integrability in time of $\|\dive u\|_{L^\infty}$. More precisely, we prove the following theorem.

\begin{theorem}\label{thm:uniqueness}
Let $(\rho_1,u_1)$ and $(\rho_2,u_2)$ be two solutions of \eqref{main system}, in the sense of \cref{def:bounded-divergence-solution}, with the same initial data $(\rho_0,u_0)$ and pressure law satisfying \eqref{pressure}. Assume in addition that $(\rho_2,u_2)$ satisfies \eqref{eq:stronger-u2-regularity}, that is, it is a Danchin--Mucha solution. Then $\rho_1=\rho_2$ and $u_1=u_2$ almost everywhere on $[0,\infty)\times\T$.
\end{theorem}

\begin{remark}
    Some comments are in order:
    \begin{itemize}
\item [(i)]
Combining \cref{thm:DM} with \cref{thm:uniqueness} yields global well posedness in the Danchin--Mucha regime: for bounded nonnegative initial density and initial velocity satisfying \eqref{eq:DM-div-smallness}, and for sufficiently large bulk viscosity, the global Danchin--Mucha solution is unique. As a byproduct, the uniqueness conjecture formulated in \cite{DM2023} is proved in two space dimensions. Additionally, only $C^1$ regularity for $P$ is required for the uniqueness part.
\item[(ii)] The uniqueness result in \cref{thm:uniqueness} holds under weaker regularity assumptions than those of the Danchin--Mucha framework. In particular, we only require $\dive u_i\in L^1_{\rm loc}([0,\infty);L^\infty(\T))$, and impose the additional bounds \eqref{eq:stronger-u2-regularity} only on $(\rho_2,u_2)$. Moreover, the energy bounds, together with the $L^1_tL^\infty_x$ control of the divergence, are critical with respect to the formal parabolic scaling in two dimensions.
\end{itemize}
\end{remark}

\subsection{Strategy of the proof}
Throughout the paper, $(\rho_1,u_1)$ denotes a solution in the sense of \cref{def:bounded-divergence-solution}, while $(\rho_2,u_2)$ denotes a Danchin--Mucha solution, which belongs to the same class and satisfies \eqref{eq:stronger-u2-regularity}. Both solutions have the same initial data. We write
\begin{align}\label{eq:notation}
P_i:=P(\rho_i),\qquad
\delta\rho:=\rho_1-\rho_2,\qquad
\delta u:=u_1-u_2,\qquad
\delta P:=P_1-P_2.
\end{align}
The kinetic relative energy formally satisfies
\begin{equation}\label{eq:kinetic-comparison-balance}
\begin{aligned}
\frac{d}{dt}\frac12\int_{\T}\rho_1|\delta u|^2\,dx
&+\mu\|\nabla\delta u\|_{L^2}^2
+(\lambda+\mu)\|\dive\delta u\|_{L^2}^2\\
&=-\int_{\T}\rho_1(\delta u\cdot\nabla u_2)\cdot\delta u\,dx-\int_{\T}\delta\rho\,\dot u_2\cdot\delta u\,dx
+\int_{\T}\delta P\,\dive\delta u\,dx.
\end{aligned}
\end{equation}
The viscous terms on the left hand side control $\|\nabla\delta u\|_{L^2}^2$. The objective is to absorb suitable gradient contributions into this dissipation and control the remaining terms by an energy functional that vanishes initially, so that a Gronwall or Osgood argument yields uniqueness.

The right hand side contains three terms, corresponding to three successive difficulties. We briefly outline the treatment of each.

\medskip
\noindent
\emph{The first term: homogeneous incompressible flow.}
When both densities equal the same positive constant $\rho_*$ and $\dive u_i=0$, only the convective term remains: the density-difference term vanishes because $\delta\rho=0$, while the pressure contribution vanishes because $\dive\delta u=0$. Its control is central to Ladyzhenskaya's classical uniqueness argument for two dimensional Leray--Hopf solutions of the homogeneous incompressible Navier--Stokes equations \cite{Ladyzhenskaya1959}. The possible presence of vacuum prevents a direct application of her argument in our setting. However, the additional regularity of the comparison solution suggests using the estimate
\begin{equation}\label{eq:IfLip}
\left|\int_{\T}\rho_1(\delta u\cdot\nabla u_2)\cdot\delta u\,dx\right|
\leq \|\sqrt{\rho_1}\,\delta u\|_{L^2}^2
\|\nabla u_2\|_{L^\infty}.
\end{equation}
If $\nabla u_2\in L_{loc}^1([0,\infty);L^\infty)$, this bound would fit directly into a Grönwall argument. In our setting, the required gradient control is available for an effective velocity obtained by correcting $u_2$ with a pressure dependent term, rather than for $u_2$ itself. This motivates the change of variables introduced in \eqref{def rel differences}; the resulting convective estimate follows the same principle as \eqref{eq:IfLip}.

\medskip
\noindent
\emph{The second term: inhomogeneous incompressible flow.}
The term $\int\delta\rho\,\dot u_2\cdot\delta u$ remains in the incompressible setting when the density is variable and transported by the velocity field. Methods for controlling it have been developed both for densities bounded away from zero \cite{HaoShaoWeiZhang2025,Skondric2025} and for densities with vacuum, including characteristic functions of sets \cite{DanchinMucha2019,SkondricViolini2026}. Taking inspiration from these works, we estimate the density difference $\delta \rho$ in negative Sobolev spaces and use duality to transfer a spatial derivative onto the product $\dot u_2\cdot\delta u$ (see \eqref{eq:T1-q-initial}).

\medskip
\noindent
\emph{The third term: compressible flow.}
The pressure contribution $\int\delta P\,\dive\delta u$ vanishes in both incompressible settings and constitutes the additional difficulty considered here. The methods used for the first two terms must therefore be supplemented by a mechanism that handles this contribution. A classical approach is to supplement the kinetic relative energy
with the relative internal energy
\begin{equation*}
h(\rho_1|\rho_2)
:=h(\rho_1)-h(\rho_2)-h'(\rho_2)\delta \rho.
\end{equation*}
For smooth solutions with positive densities, the continuity
equations gives
\begin{equation}\label{eq:diffh}
\begin{split}
\frac{d}{dt}\int_{\T}h(\rho_1|\rho_2)\,dx
={}&-\int_{\T}\delta P\,\dive\delta u\,dx-\int_{\T}
\bigl[\delta P-P'(\rho_2)\delta\rho\bigr]\dive u_2\,dx
-\int_{\T}\delta\rho\,\nabla h'(\rho_2)\cdot\delta u\,dx.
\end{split}
\end{equation}
Thus, adding \eqref{eq:diffh} to
\eqref{eq:kinetic-comparison-balance} cancels the term
$\int \delta P\,\dive\delta u $.
However, even when considering increasing pressure laws, the remaining terms introduce two difficulties. First, the last integral involves $\nabla h'(\rho_2)$, which cannot be controlled in our regularity framework.Second, even if $P$ is convex, its $C^1$ regularity and the boundedness of the densities do not in general imply the estimate
\begin{equation*}
|P(\rho_1|\rho_2)|\le C h(\rho_1|\rho_2),
\qquad
P(\rho_1|\rho_2):=\delta P-P'(\rho_2)\delta\rho,
\end{equation*}
which holds for the power laws \eqref{press Danchin-Mucha}.
For nonmonotone pressure laws, a further difficulty arises since  $h$ need not be convex and thus the relative internal energy may be negative. These difficulties prevent a direct application of the classical relative energy argument. We therefore pursue a kinetic relative-energy approach and seek a different way to control the pressure contribution $\int \delta P\,\dive\delta u$. Two obstacles arise.

First, estimating the pressure difference in $L^2$ gives
\[
\left|\int_{\T}\delta P\,\dive\delta u\,dx\right|
\leq\varepsilon\|\nabla\delta u\|_{L^2}^2
+C_\varepsilon\|\delta P\|_{L^2}^2,
\]
but the kinetic relative energy does not control $\|\delta P\|_{L^2}^2$.

Second, estimating $\delta P$ in $H^{-1}$ would instead require control of $\dive\delta u$ in $H^1$, hence second spatial derivatives of the velocity. Such regularity is not supplied by the momentum equation when the density, and therefore the pressure, may be discontinuous.

Our strategy first separates the pressure difference into a compression part controlled in $L^2$ and a displacement part controlled in $H^{-1}$. This splitting then guides the choice of the corrected velocities: the displacement contribution is incorporated into the velocity difference, while the compression contribution remains as a pressure force that can be absorbed into the dissipation.

\medskip
\noindent
\emph{Separating compression from displacement.}
Let $X_i$ be the flow of $u_i$ and $J_i$ its Jacobian factor. Our aim is to control $\delta P$ in $H^{-1}$ through the displacement and compression of the two flows. Motivated by the continuity identity \eqref{eq:lagrangian-density}, we introduce
\begin{equation}\label{def x and j}
x^2(t):=\int_{\T}|X_1(t,\xi)-X_2(t,\xi)|^2\rho_0(\xi)\,d\xi,\qquad
j^2(t):=\int_{\T}|J_1(t,\xi)-J_2(t,\xi)|^2\rho_0(\xi)\,d\xi.
\end{equation}
To see how these quantities enter the estimate, take $\varphi\in H^1(\T)$, change variables along each flow, and add and subtract the crossed term $\int_{\T}P_2(X_2)J_2\varphi(X_1)\,d\xi$. Suppressing the time variable, we obtain
\begin{align*}
\int_{\T}\delta P\,\varphi\,dx
&=\int_{\T}P_1(X_1)J_1\varphi(X_1)\,d\xi
-\int_{\T}P_2(X_2)J_2\varphi(X_2)\,d\xi\\
&=\int_{\T}\bigl[P_1(X_1)J_1-P_2(X_2)J_2\bigr]\varphi(X_1)\,d\xi+\int_{\T}P_2(X_2)J_2\bigl[\varphi(X_1)-\varphi(X_2)\bigr]\,d\xi.
\end{align*}
By \eqref{eq:lagrangian-density}, the pressure weights in the first integral depend on the same initial density and on the respective Jacobians. This contribution therefore measures the difference in compression and is bounded by $Cj\|\varphi\|_{L^2}$. In the second integral, the pressure weight is fixed and only the position of the test function changes. Combining the transport argument of \cite{LOEPER200668} with the maximal function estimates of \cite{CripDel08}, we bound this contribution by $Cx\|\nabla\varphi\|_{L^2}$.

To simplify the exposition, we give the crossed pressure weight an Eulerian representation by defining the intermediate pressure $\widetilde P$ through
\begin{equation}\label{int pressure}
\widetilde P(t,X_1(t,\xi))J_1(t,\xi)
=P_2(t,X_2(t,\xi))J_2(t,\xi).
\end{equation}
The crossed term then equals $\int \widetilde P\,\varphi $, and the two contributions above become $\int(P_1-\widetilde P)\varphi $ and $\int(\widetilde P-P_2)\varphi$, respectively. Thus
\[
\delta P=(P_1-\widetilde P)+(\widetilde P-P_2)
\]
expresses the compression and displacement contributions in Eulerian form. More precisely, \cref{lem:stability} gives
\begin{equation}\label{eq:uniq-overview-pressure-stability}
\|P_1-\widetilde P\|_{L^2}\lesssim j,\qquad
\|\widetilde P-P_2\|_{H^{-1}}\lesssim x.
\end{equation}
In particular, $\|\delta P\|_{H^{-1}}\lesssim x+j$, while the separate estimates retain the stronger $L^2$ control of the compression contribution.

\medskip
\noindent
\emph{Choosing the corrected velocities.}
We now use this splitting to address the lack of second spatial derivatives of the velocity. Let $\mathbb P$ be the Leray projection and set $G_i:=\nu\dive u_i-P_i$, where $\nu=\lambda+2\mu$. The elliptic estimates associated with the momentum equation give
\[
\|\nabla^2\mathbb P u_i\|_{L^2}
+\|\nabla G_i\|_{L^2}
\lesssim \|\rho_i\dot u_i\|_{L^2},
\]
whenever the right hand side is finite. These estimates control the solenoidal part of the velocity and the effective viscous flux, but not $\nabla^2u_i$: recovering the full velocity would require a derivative of $P_i$. This motivates a pressure correction, following the effective velocity approach of \cite{Haspot2011}.

Define
\[
\mathcal Rf:=\nabla\Delta^{-1}(f-\bar f),\qquad
\Pi_2:=\mathcal RP_2,\qquad
\widetilde\Pi:=\mathcal R\widetilde P,
\]
where $\bar f$ denotes the spatial mean, and set
\begin{equation}\label{def rel differences}
w_2:=u_2-\frac1\nu\Pi_2,\qquad
w_1:=u_1-\frac1\nu\widetilde\Pi,\qquad
\delta w:=w_1-w_2.
\end{equation}
These two corrections serve different purposes. For the Danchin--Mucha solution, subtracting $\nu^{-1}\Pi_2$ cancels the pressure gradient in the elliptic equation:
\[
\mu\Delta w_2+(\lambda+\mu)\nabla\dive w_2
=\rho_2\dot u_2.
\]
Consequently,
\[
\|\nabla^2w_2\|_{L^q}
\leq C_q\|\rho_2\dot u_2\|_{L^q},
\qquad 2<q<\infty;
\]
see \cref{lem:elliptic-estimates}. In particular, the regularity of $\dot u_2$ provides control of $\nabla w_2$ in $L^\infty$ without requiring spatial derivatives of the pressure. The corrected relative energy argument therefore involves $\nabla w_2$ in place of $\nabla u_2$.

For the first solution, using $\widetilde P$ in the correction precisely incorporates the displacement part of the pressure difference into the velocity difference. Indeed,
\begin{equation}\label{eq:deltawDef}
\delta u=\delta w+\frac1\nu\mathcal R(\widetilde P-P_2),
\qquad
\|\delta u-\delta w\|_{L^2}\lesssim x,
\end{equation}
by \eqref{eq:uniq-overview-pressure-stability} and the boundedness of $\mathcal R:H^{-1}\to L^2$. The compression part $P_1-\widetilde P$ remains in the pressure force, where its $L^2$ control allows us to pair it with a single derivative of $\delta w$.

More precisely, writing $D_i:=\partial_t+u_i\cdot\nabla$, \cref{lemma:rel-energy} gives
\begin{equation}\label{eq:uniq-overview-relative-energy}
\begin{aligned}
\frac12\frac{d}{dt}\|\sqrt{\rho_1}\,\delta w\|_{L^2}^2
+c\|\nabla\delta w\|_{L^2}^2
&\leq -\int_{\T}\delta\rho\,\dot u_2\cdot\delta w\,dx -\int_{\T}\rho_1(\delta u\cdot\nabla w_2)\cdot\delta w\,dx\\
&\quad+\int_{\T}(P_1-\widetilde P)\dive\delta w\,dx-\frac1\nu\int_{\T}\rho_1(D_1\widetilde\Pi-D_2\Pi_2)\cdot\delta w\,dx,
\end{aligned}
\end{equation}
for some $c>0$ depending on the viscosities. The remaining pressure force is now estimated directly:
\[
\left|\int_{\T}(P_1-\widetilde P)\dive\delta w\,dx\right|
\lesssim j\|\nabla\delta w\|_{L^2}
\leq \varepsilon\|\nabla\delta w\|_{L^2}^2+C_\varepsilon j^2.
\]
The density and convective terms are treated by relative energy arguments using the regularity of $\dot u_2$ and $w_2$, together with \eqref{eq:deltawDef}. The last term is the additional cost of the pressure correction. Its material derivatives are decomposed in \eqref{eq:direct-material-splitting} and estimated in \cref{lemma:G2}, giving
\[
\left|\frac1\nu\int_{\T}\rho_1(D_1\widetilde\Pi-D_2\Pi_2)\cdot\delta w\,dx\right|
\leq \varepsilon\|\nabla\delta w\|_{L^2}^2
+C_\varepsilon\|\sqrt{\rho_1}\,\delta w\|_{L^2}^2
+C_\varepsilon\bigl(1+\|\sqrt{\rho_2}\dot u_2\|_{L^2}^2\bigr)x^2.
\]

It remains to control the two flow quantities $x$ and $j$ and close the argument. By \cref{lem:flowStability},
\begin{align}
\frac{d}{dt}x^2
\lesssim{}&
x\|\sqrt{\rho_1}\,\delta w\|_{L^2}+x^2
+\bigl(1+\|\nabla w_2\|_{L^\infty}\bigr)x^2\log\frac{e}{x^2},
\label{eq:uniq-overview-displacement}\\
\frac{d}{dt}j^2
\lesssim{}&
\varepsilon\|\nabla\delta w\|_{L^2}^2
+\bigl(1+\|G_1\|_{L^\infty}\bigr)j^2
+\|\dot u_2\|_{H^1}\,jx\left(\log\frac{e}{x^2}\right)^{3/2}.
\label{eq:uniq-overview-compression}
\end{align}
We must now combine the relative energy inequality with the estimates for $x^2$ and $j^2$. The main obstruction to simply adding these three inequalities is the last term in \eqref{eq:uniq-overview-compression}. Indeed, setting
\[
L:=\|\sqrt{\rho_1}\,\delta w\|_{L^2}^2+x^2+j^2,
\]
we obtain, for sufficiently small $L$,
\begin{equation}\label{eq:32}
jx\left(\log\frac{e}{x^2}\right)^{3/2}
\lesssim L\left(\log\frac{e}{L}\right)^{3/2}.
\end{equation}
The resulting modulus does not satisfy the Osgood criterion, since
\[
\int_0^\delta
\frac{ds}{s\bigl(\log(e/s)\bigr)^{3/2}}<\infty
\]
for every sufficiently small $\delta>0$. The exponent $3/2$ in \eqref{eq:32} could be reduced to $1$ if one could establish
\[
\sqrt{t}\,G_2\in L_{\mathrm{loc}}^1([0,\infty);LL),
\]
where $LL$ denotes the space of log-Lipschitz functions. This condition has the same scaling as the known regularity $u_2\in L^1(0,T;LL)$; see \cite[Corollary~1.1]{DM2023}. However, we do not know how to obtain the corresponding estimate for $G_2$. The logarithmic exponent greater than $1$ therefore prevents us from closing a direct Osgood argument for $L$.

A further difficulty arises near the initial time: the available regularity is $\sqrt{t}\,\dot u_2\in L^2(0,T;H^1)$ (see \eqref{eq:acceleration-Poincare}, which does not imply $\dot u_2\in L^1(0,T;H^1)$. Although these two conditions have the same scaling in time, the latter cannot be deduced by the Cauchy--Schwarz inequality, since $\int_0^T t^{-1}\,dt=\infty$.

To handle both issues, we use the time-dependent functional
\[
F(t):=\|\sqrt{\rho_1}\,\delta w(t)\|_{L^2}^2+j^2(t)
+\frac{x^2(t)}{t}\log\frac{e}{x^2(t)},
\]
for sufficiently small $t>0$, with the last term defined to be zero when $x(t)=0$. The weighted displacement term compensates for the initial time singularity and allows the mixed compression displacement term to be estimated with an Osgood modulus and an integrable time dependent coefficient. Section~4 establishes the resulting differential inequality, verifies that $F(t)\to0$ as $t\downarrow0$, and concludes that $F\equiv0$, yielding uniqueness.

\subsection{Notation and conventions}
Beyond the standard notation used in the literature, we assume the following conventions. Given two functions $f$ and $g,$ the symbol $f \lesssim g$ denotes the existence of a generic constant $c>0$ such that $f \le c g.$ Unless otherwise specified, all the considered integrals are taken over $\T.$ For simplicity we assume $|\T|=1,$ while $dx$ is omitted from spatial integrals. 
We fix an arbitrary time $T>0$. All implicit constants
may depend on $T$. Shorter time intervals will be introduced explicitly when needed. Products involving logarithmic factors are interpreted at zero by continuous extension whenever this extension exists.
\subsection{AI declaration}
The authors asked a large language model (LLM) whether an approach inspired by \cite{HanKwanMiotMoussaMoyano2020} could be adapted to estimate the pressure difference $\|P_1-P_2\|_{H^{-1}}$ for nonlinear pressure laws. In response, the LLM generated from scratch an initial version of the statement and proof of \cref{lem:stability}, introducing the quantity $j$, which plays a key role in the analysis. The authors subsequently revised the statement and simplified the proof by introducing the additional decomposition involving $\widetilde P$. The authors take full responsibility for the content of this work.
\subsection{Organization of the paper}
The outline of the paper is the following. Section \ref{sec pressure} is devoted to the derivation of suitable pressure estimates, together with the analysis of the functionals $x$ and $j.$ In Section \ref{sec relative energy} we establish a relative energy inequality. Section \ref{sec proof of the main result} contains the proof of our main result, \cref{thm:uniqueness}.
\section{Pressure estimates}\label{sec pressure}
We start our analysis with the following lemma providing estimates of pressure differences in negative Sobolev spaces.
\begin{lemma}\label{lem:stability}
Assume the pressure law $P(\rho)$ satisfies \eqref{pressure} and let the intermediate pressure $\widetilde{P}$ be defined by \eqref{int pressure}. Then the following estimates hold for almost every $t\in(0,T)$.
\begin{align}\label{eq:tildastab}
\|P_1-\widetilde P\|_{L^2}
\lesssim j,
\qquad
\|\widetilde P-P_2\|_{H^{-1}}
\lesssim x,
\end{align}
and
\begin{align}\label{eq:rhoStab}
\|\rho_1-\rho_2\|_{H^{-1}}
\lesssim x.
\end{align}
The implicit constants depend on the uniform $L^\infty$ bounds for
$\rho_1,\rho_2$, on the uniform upper and lower bounds for $J_1,J_2$, and on $P$ through its $C^1$.
\end{lemma}

\begin{proof}
We first show how \eqref{eq:tildastab} implies \eqref{eq:rhoStab}. If $P(\rho)=\rho$, then $\widetilde P=\rho_1$ therefore \eqref{eq:rhoStab} follows from the second estimate on \eqref{eq:tildastab}. It remains to prove \eqref{eq:tildastab}. We begin with $P_1-\widetilde P$. Conservation of mass gives
\begin{equation}\label{eq:massT}
\rho_i(X_i)J_i=\rho_0.
\end{equation}
Hence, by the definition of $\widetilde P$,
\[
J_1(P_1-\widetilde P)(X_1)
=
J_1P\left(\frac{\rho_0}{J_1}\right)
-
J_2P\left(\frac{\rho_0}{J_2}\right).
\]
For each particle label, we hold $\rho_0$ fixed. Since $P(0)=0$ and $P$ is $C^1$, we have $|P(r)-rP'(r)|\leq Cr$. Differentiating with respect to $J$ and using the uniform lower bound on the Jacobians, we obtain
\[
\left|\frac{d}{dJ}
\left[JP\left(\frac{\rho_0}{J}\right)\right]\right|
=
\left|P\left(\frac{\rho_0}{J}\right)
-\frac{\rho_0}{J}P'\left(\frac{\rho_0}{J}\right)\right|
\lesssim \rho_0.
\]
The mean value theorem therefore yields
\begin{align}\label{eq:noConv}
    |J_1(P_1-\widetilde P)(X_1)|  = \left| J_1P\left(\frac{\rho_0}{J_1}\right)
-
J_2P\left(\frac{\rho_0}{J_2}\right)\right|\lesssim \rho_0|J_1-J_2|.
\end{align}
Changing variables through $X_1$ and using $\rho_0/J_1=\rho_1(X_1)$, we conclude that
\[
\begin{aligned}
\|P_1-\widetilde P\|_{L^2}^2
=\int 
\frac{|J_1(P_1-\widetilde P)(X_1)|^2}{J_1} \lesssim\int 
\rho_0\,\rho_1(X_1)|J_1-J_2|^2 
\leq \|\rho_1\|_{L^\infty}j^2.
\end{aligned}
\]
This proves the first estimate in \eqref{eq:tildastab}. We estimate $\widetilde P-P_2$. For any $\varphi\in H^1$, by the definition of $\widetilde P$,
\begin{align*}
\int
(\widetilde P-P_2)\varphi\,dx
&=
\int
P_2(X_2)J_2
\bigl[\varphi(X_1)-\varphi(X_2)\bigr]\,d\xi 
\end{align*}
Since $P(0)=0$ and the density is uniformly bounded,
\[
|P_2(X_2)J_2|
=
\left|
J_2P\left(\frac{\rho_0}{J_2}\right)
\right|
\lesssim
\rho_0.
\]
Hence, using the Sobolev maximal function inequality,
\begin{align*}
\left|
\int
(\widetilde P-P_2)\varphi\,dx
\right|
&\lesssim
\int
\rho_0|X_1-X_2|
\Bigl(
\mathcal M(|\nabla\varphi|)(X_1)
+
\mathcal M(|\nabla\varphi|)(X_2)
\Bigr)\,d\xi
\\
&\lesssim
\|\sqrt{\rho_0}(X_1-X_2)\|_{L^2}
\sum_{i=1}^2
\|\sqrt{\rho_0}\,
\mathcal M(|\nabla\varphi|)(X_i)\|_{L^2}\lesssim
x\,\|\nabla\varphi\|_{L^2}.
\end{align*}
which proves \eqref{eq:tildastab} and completes the proof.
\end{proof}
We continue our analysis with the following lemma.
\begin{lemma}\label{lem:flowStability}
For every $\varepsilon>0$, the following inequalities hold
for almost every $t\in(0,T)$ such that $x(t)\le e^{-2}$. 
\begin{align}
\frac{d}{dt}x^2
\lesssim\;&
x\|\sqrt{\rho_1}\,\delta w\|_{L^2}
+x^2
+
\bigl(1+\|\nabla w_2\|_{L^\infty}\bigr)
x^2\log\frac{e}{x^2},
\label{eq:xStability}
\\
\frac{d}{dt}j^2
\lesssim\;&
\varepsilon\|\nabla\delta w\|_{L^2}^2
+
\bigl(1+\|G_1\|_{L^\infty}\bigr)
j^2+
\|\dot u_2\|_{H^1}\,
j\,x
\left(
\log\frac{e}{x^2}
\right)^{3/2}.
\label{eq:jStability}
\end{align}
\end{lemma}

\begin{proof}
We first prove \eqref{eq:xStability}.
Differentiating the definition of $x^2$ and using $\partial_tX_i=u_i(X_i)$, we obtain
\begin{align*}
\frac12\frac{d}{dt}x^2
&=
\int
\rho_0\,\delta X\cdot
\bigl(u_1(X_1)-u_2(X_2)\bigr)
\\
&=
\int
\rho_0\,\delta X\cdot\delta u(X_1)
+
\int
\rho_0\,\delta X\cdot
\bigl(u_2(X_1)-u_2(X_2)\bigr)
=:S_1+S_2.
\end{align*}
Combining the Cauchy--Schwarz inequality with the definition of $w$ (see \eqref{eq:deltawDef}), we obtain
\begin{align*}
|S_1|
&\lesssim
\|\sqrt{\rho_0}\,\delta X\|_{L^2}
\|\sqrt{\rho_0}\,\delta u(X_1)\|_{L^2}
\lesssim
x
\left(
\|\sqrt{\rho_1}\,\delta w\|_{L^2}
+
\|\widetilde\Pi-\Pi_2\|_{L^2}
\right)\lesssim
x
\left(
\|\sqrt{\rho_1}\,\delta w\|_{L^2}
+x
\right),
\end{align*}
where in the last inequality we have used \eqref{eq:tildastab} and the continuity of $\mathcal{R}:H^{-1}\rightarrow L^2$. To estimate $S_2$, we recall that, by \eqref{eq:u2-loglip},
\[
|u_2(y)-u_2(z)|
\lesssim
\bigl(1+\|\nabla w_2\|_{L^\infty}\bigr)
|y-z|
\log\frac{e}{|y-z|}.
\]
Consequently, using the concavity of the function $ s\longmapsto s\log\frac{e}{\sqrt{s}}$ and the fact that $\rho_0$ has unit mass (see \eqref{eq:unitmass}), we obtain 
\begin{align*}
|S_2|
&\lesssim
\bigl(1+\|\nabla w_2\|_{L^\infty}\bigr)
\int
\rho_0\,|\delta X|^2
\log\frac{e}{|\delta X|}
\lesssim
\bigl(1+\|\nabla w_2\|_{L^\infty}\bigr)
x^2\log\frac{e}{x^2}.
\end{align*}
Combining the estimates for $S_1$ and $S_2$ yields
\eqref{eq:xStability}.

We now turn to $j$. We combine \eqref{eq:massT} with $J_i'
= J_i\,\operatorname{div}u_i(X_i)$ and $ G_i=\nu\operatorname{div}u_i-P_i,$ to obtain
\[
J_i'
=
\frac1\nu
J_iP\left(\frac{\rho_0}{J_i}\right)
+
\frac1\nu G_i(X_i)J_i.
\]
It follows that
\begin{align*}
\frac12\frac{d}{dt}j^2
&=\frac1\nu\int\rho_0\,\delta J
\left[
J_1P\left(\frac{\rho_0}{J_1}\right)
-J_2P\left(\frac{\rho_0}{J_2}\right)
\right] +\frac1\nu\int\rho_0\,\delta J
\bigl(G_1(X_1)J_1-G_2(X_2)J_2\bigr) \\ & 
=:T_1+T_2.
\end{align*}
By \eqref{eq:noConv} and using that $\rho_0\in L^\infty$ we immediately get
\[
|T_1|
\lesssim\int\rho_0^2|\delta J|^2
\le\|\rho_0\|_{L^\infty}\int\rho_0|\delta J|^2
\lesssim j^2.
\]
For $T_2$ we have (up to the harmless factor $\nu^{-1}$),
\begin{align*}
T_2
&=
\int
\rho_0\; \delta J \; 
G_1(X_1)\; \delta J +
\int
\rho_0\; \delta J  \;
J_2\,\delta G(X_1)
+
\int
\rho_0\; \delta J \;
J_2\bigl(G_2(X_1)-G_2(X_2)\bigr)
\\
&=:T_2^{(1)}+T_2^{(2)}+T_2^{(3)}.
\end{align*}
For the first term, $|T_2^{(1)}|
\lesssim
\|G_1\|_{L^\infty}j^2. $  For $T_2^{(2)}$, we have  $\delta G
=
\nu\operatorname{div}\delta w
- (P_1 - \widetilde P)$ and therefore, by \cref{lem:stability}
\begin{align*}
\|\delta G\|_{L^2}
&\lesssim
\|\nabla\delta w\|_{L^2}
+ \| P_1 - \widetilde P\|_{L^2}
\lesssim
\|\nabla\delta w\|_{L^2}
+j,
\end{align*}
Hence,
\[
|T_2^{(2)}|
\lesssim
j\bigl(
\|\nabla\delta w\|_{L^2}+j
\bigr) \le
\varepsilon\|\nabla\delta w\|_{L^2}^2
+
C_\varepsilon j^2.
\]
Finally, we estimate $T_2^{(3)}$. We claim that
\begin{align}\label{eq:T_2^3}
\|\sqrt{\rho_0}
(G_2(X_1)-G_2(X_2))\|_{L^2}
\lesssim
\|\dot u_2\|_{H^1}\,
x
\left(
\log\frac{e}{x^2}
\right)^{3/2}.
\end{align}
Assume \eqref{eq:T_2^3}, then by Cauchy--Schwarz
\[
|T_2^{(3)}|
\lesssim
\|\dot u_2\|_{H^1}\,
j\,x
\left(
\log\frac{e}{x^2}
\right)^{3/2}.
\]
Combining the estimates for $T_2^{(1)}$, $T_2^{(2)}$ and $T_2^{(3)}$ we get \eqref{eq:jStability}. To conclude the proof we are left to show \eqref{eq:T_2^3}, we first use the maximal-function inequality:
\[
|G_2(X_1)-G_2(X_2)|
\lesssim
|\delta X|
\Bigl(
\mathcal M(|\nabla G_2|)(X_1)
+
\mathcal M(|\nabla G_2|)(X_2)
\Bigr),
\]
which leads to 
\[
\|\sqrt{\rho_0}
(G_2(X_1)-G_2(X_2))\|_{L^2}
\lesssim \|\sqrt{\rho_0}
\; \delta X\|_{L^{\frac{2q}{q-2}}} \;
\|\nabla G_2\|_{L^q}
\]
for every $q>2$.
For the first term on the right-hand side, we interpolate $L^{\frac{2q}{q-2}}
=
[L^2,L^\infty]_\theta$, $\theta=\frac{2}{q}.$ Since we are on $\mathbb T^2$, we have $
\|\delta X\|_{L^\infty}\lesssim 1$,
and therefore
\[
\| \sqrt{\rho_0}\delta X\|_{L^{\frac{2q}{q-2}}}
\lesssim
\|\sqrt{\rho_0}\delta X\|_{L^2}^{\,\theta}
\|\sqrt{\rho_0}\delta X\|_{L^\infty}^{\,1-\theta}
\lesssim
x^{1-\frac2q}.
\]
To estimate $ \|\nabla G_2\|_{L^q}$ we recall that the effective viscous flux solves 
\[
\Delta G_2=\dive (\rho_2 \dot u_2),
\]
Since the $L^q$ norm of the Calder\'on--Zygmund operator $\nabla \Delta^{-1} \dive $ is bounded by $Cq$ for $q\ge2$ we get
\begin{align*}
\|\nabla G_2\|_{L^q}
&\lesssim
q\,\|\rho_2\dot u_2\|_{L^q}\lesssim
q^{3/2}\|\dot u_2\|_{H^1}.
\end{align*}
where in the last line we have used $\| f\|_{L^q} \lesssim \sqrt{q} \| f\|_{H^1}$ and $\rho_2$ is bounded.

Inserting the two estimates we have 
\begin{align}\label{eq:lastJest}
\|\sqrt{\rho_0}
(G_2(X_1)-G_2(X_2))\|_{L^2}
\lesssim x^{1-\frac2q} q^{3/2}\|\dot u_2\|_{H^1}.
\end{align}
We now minimize the function $q\longmapsto x^{1-\frac{2}{q}}q^{\frac32} $
for fixed $0<x<1$. The minimum is attained at $q=-\frac43\log x.$ Under the condition $0<x\le e^{-2}$ in the statement,
this choice satisfies $q>2$. The resulting estimate extends to $x=0$ by continuity.
\end{proof}

\section{Relative energy}\label{sec relative energy}
In this section we develop a relative energy framework. For simplicity of presentation, \cref{lemma:rel-energy} contains the relative energy inequality and the estimates of the first term in the right hand side of \eqref{eq:rel-energy}. The estimate of the second term is significantly more involved and the proof is therefore postponed to \cref{lemma:G2}.
\begin{lemma}\label{lemma:rel-energy}
The following relative energy inequality holds
for almost every $t\in(0,T)$:
\begin{equation}\label{eq:rel-energy}
\begin{split}
\frac12\frac{d}{dt}\|\sqrt{\rho_1}\delta w\|_{L^2}^2
+c\|\nabla\delta w\|_{L^2}^2
&\le \mathcal{G}_1+\mathcal{G}_2,
\end{split}
\end{equation}
where $c=\min(\mu,\nu)>0$ and
\begin{align}\label{eq:Gdef}
\begin{split}
\mathcal{G}_1
&:=
-\int \delta\rho\,\dot u_2\cdot\delta w
-\int \rho_1(\delta u\cdot\nabla w_2)\cdot\delta w
+\int (P_1-\widetilde P)\dive\delta w,
\\
\mathcal{G}_2
&:=
-\frac1\nu\int\rho_1
(D_1\widetilde\Pi-D_2\Pi_2)\cdot\delta w.
\end{split}
\end{align}
Moreover, for every $\varepsilon>0$, the following estimate holds for almost every $t\in(0,T)$ such that $x(t)\le e^{-2}$:
\begin{align}\label{eq:G1-estimate}
\begin{split}
|\mathcal{G}_1|
\le
\varepsilon\|\nabla\delta w\|_{L^2}^2
+C ( \|\nabla w_2\|_{L^\infty} \; x\; \|\sqrt{\rho_1}\delta w\|_{L^2}+ (1+ \|\nabla w_2\|_{L^\infty})
\|\sqrt{\rho_1}\delta w\|_{L^2}^2
+
\|\dot u_2\|_{H^1}^2
x^2\log\frac e{x^2}
+
j^2).
\end{split}
\end{align}
The constants depend on the uniform density and Jacobian bounds, the viscosities, the $C^1$ norm of $P$ and $\varepsilon.$
\end{lemma}

\begin{proof}

We first perform a formal computation. Recall that $w_1$ and $w_2$ satisfy
\[
\mu\Delta w_1+(\lambda+\mu)\nabla\dive w_1
=
\rho_1\dot u_1+\nabla(P_1-\widetilde P),
\qquad
\mu\Delta w_2+(\lambda+\mu)\nabla\dive w_2
=
\rho_2\dot u_2.
\]
Subtracting the equations satisfied by $w_1$ and $w_2$, and writing
$D_i:=\partial_t+u_i\cdot\nabla$, we obtain
\begin{align*}
\rho_1 D_1\delta w
-\mu\Delta\delta w
-(\lambda+\mu)\nabla\dive\delta w
={}&
-\delta\rho\,\dot u_2
-\rho_1(\delta u\cdot\nabla)w_2
-\nabla(P_1-\widetilde P)-\frac{\rho_1}{\nu}
\bigl(D_1\widetilde\Pi-D_2\Pi_2\bigr).
\end{align*}
Taking the $L^2$ scalar product with $\delta w$ gives
\begin{align*}
\int \rho_1 D_1\delta w\cdot\delta w
&+
\mu\int |\nabla\delta w|^2
+
(\lambda+\mu)\int |\dive\delta w|^2
=
-\int \delta\rho\,\dot u_2\cdot\delta w
-\int \rho_1(\delta u\cdot\nabla w_2)\cdot\delta w
\\
&+\int(P_1-\widetilde P)\dive\delta w-\frac1\nu\int\rho_1
\bigl(D_1\widetilde\Pi-D_2\Pi_2\bigr)\cdot\delta w.
\end{align*}
For the first term, we have
\[
\int\rho_1D_1\delta w\cdot\delta w
=
\frac12\frac{d}{dt}
\|\sqrt{\rho_1}\delta w\|_{L^2}^2.
\]
Moreover, recalling that $\nu=\lambda+2\mu$, we have
\[
\mu\|\nabla\delta w\|_{L^2}^2
+
(\lambda+\mu)\|\dive\delta w\|_{L^2}^2
=
\mu\|\operatorname{curl}\delta w\|_{L^2}^2
+
\nu\|\dive\delta w\|_{L^2}^2
\ge
c\|\nabla\delta w\|_{L^2}^2,
\]
where $c=\min(\mu,\nu)>0$. 

The same identity can be obtained rigorously from the energy inequality and the weak formulations of the momentum and continuity equations, reproducing the result of the above formal computation. To this end, we expand
\[
\frac12\|\sqrt{\rho_1}\delta w\|_{L^2}^2
=
\frac12\int_{\T}\rho_1|u_1|^2\,dx
-\int_{\T}\rho_1u_1\cdot\left(w_2+\frac{\widetilde\Pi}{\nu}\right)\,dx
+\frac12\int_{\T}\rho_1\left|w_2+\frac{\widetilde\Pi}{\nu}\right|^2\,dx.
\]
The first term is handled by the kinetic energy balance, and the other two by testing the momentum and continuity equations with $w_2+\widetilde\Pi/\nu$ and $\frac12|w_2+\widetilde\Pi/\nu|^2$, respectively. The admissibility of these tests and the passage to the limit in their regularized versions follow by the same argument as in \cite[Section 3]{CrinBaratSkondricViolini2025}. Combining the resulting identities yields precisely the time integrated version of the identity derived above. The continuity
\begin{align}\label{eq:contdeltaw}
\sqrt{\rho_1}\delta w\in C([0,\infty);L^2(\T))
\end{align}
can be established by arguing as in \cite[Appendix]{SkondricViolini2026}. Together with the integrability of the terms in the relative energy identity, this allows the lower time endpoint to tend to zero. The resulting time integrated identity also shows that $t\mapsto\frac12\|\sqrt{\rho_1}\delta w(t)\|_{L^2}^2$ is locally absolutely continuous, so that its differential form holds for almost every time. We omit the details of this justification and use the differential form for readability.

We now prove \eqref{eq:G1-estimate}, starting from the first term.
Conservation of mass implies that $\delta\rho$ has zero mean, and
therefore, for every $q>2$,
\begin{align}\label{eq:T1-q-initial}
\begin{split}
\left|
\int \delta\rho\,\dot u_2\cdot\delta w
\right|
&=
\left|
\int \dive\mathcal R\delta\rho\,
\dot u_2\cdot\delta w
\right|=
\left|
\int \mathcal R\delta\rho\cdot
\nabla(\dot u_2\cdot\delta w)
\right|
\\
&\le
\|\mathcal R\delta\rho\|_{L^{\frac{2q}{q-2}}}
\left(
\|\delta w\|_{L^q}\|\nabla\dot u_2\|_{L^2}
+
\|\dot u_2\|_{L^q}\|\nabla\delta w\|_{L^2}
\right).
\end{split}
\end{align}
Now,
\[
\|\mathcal R\delta\rho\|_{L^{\frac{2q}{q-2}}}
\lesssim
\|\mathcal R\delta\rho\|_{L^2}^{1-\frac2q}
\|\mathcal R\delta\rho\|_{L^\infty}^{\frac2q}
\lesssim
\|\delta\rho\|_{H^{-1}}^{1-\frac2q}
\|\mathcal R\delta\rho\|_{L^\infty}^{\frac2q}
\lesssim
x^{1-\frac2q},
\]
where we used \eqref{eq:rhoStab} and $\|\mathcal R\delta\rho\|_{L^\infty} \lesssim\|\delta\rho\|_{L^\infty}\lesssim1$.
Moreover, by Sobolev embedding,
$\|f\|_{L^q}\le C\sqrt q\,\|f\|_{H^1}$, and therefore
\begin{align}\label{eq:T1-q-bound}
\left|
\int \delta\rho\,\dot u_2\cdot\delta w
\right|
\lesssim
\sqrt q\,x^{1-\frac2q}
\|\dot u_2\|_{H^1}\|\delta w\|_{H^1}.
\end{align}
Arguing as in the estimate of \eqref{eq:lastJest}, in our range of $x$ the function $q\mapsto\sqrt q\,x^{1-\frac2q}$ on $(2,\infty)$ attains its minimum at $q=4\log(1/x)$. Inserting this value
into \eqref{eq:T1-q-bound}, and using the weighted Poincar\'e inequality and Young's inequality, we obtain
\begin{align}\label{eq:T1-final}
\begin{split}
\left|
\int \delta\rho\,\dot u_2\cdot\delta w
\right|
&\lesssim
\|\dot u_2\|_{H^1}
x\sqrt{\log\frac e{x^2}}\,
\|\delta w\|_{H^1}
\\
&\lesssim
\|\dot u_2\|_{H^1}
x\sqrt{\log\frac e{x^2}}
\left(
\|\sqrt{\rho_1}\delta w\|_{L^2}
+
\|\nabla\delta w\|_{L^2}
\right)
\\
&\le
\varepsilon\|\nabla\delta w\|_{L^2}^2
+
C\|\sqrt{\rho_1}\delta w\|_{L^2}^2
+
C\|\dot u_2\|_{H^1}^2
x^2\log\frac e{x^2}.
\end{split}
\end{align}
For the second term, the velocity difference satisfies \eqref{eq:deltawDef} and so, by \eqref{eq:tildastab}
\begin{align}\label{eq:T2-estimate}
\begin{split}
\left|
\int \rho_1(\delta u\cdot\nabla w_2)\cdot\delta w
\right|
&\le
\|\nabla w_2\|_{L^\infty}
\|\sqrt{\rho_1}\delta u\|_{L^2}
\|\sqrt{\rho_1}\delta w\|_{L^2}
\\
&\le
C\|\nabla w_2\|_{L^\infty}
\|\sqrt{\rho_1}\delta w\|_{L^2}
\left(
\|\sqrt{\rho_1}\delta w\|_{L^2}
+
\|\widetilde\Pi-\Pi_2\|_{L^2}
\right)
\\
&\le
C\|\nabla w_2\|_{L^\infty}
\|\sqrt{\rho_1}\delta w\|_{L^2}
\left(
\|\sqrt{\rho_1}\delta w\|_{L^2}
+
\|\widetilde P-P_2\|_{H^{-1}}
\right)
\\
&\le
C\|\nabla w_2\|_{L^\infty}
\|\sqrt{\rho_1}\delta w\|_{L^2}
\left(
\|\sqrt{\rho_1}\delta w\|_{L^2}
+x
\right).
\end{split}
\end{align}
The last term is controlled directly by \eqref{eq:tildastab}:
\begin{align}\label{eq:T4-estimate}
\left|
\int(P_1-\widetilde P)\dive\delta w
\right|
&\le
\|P_1-\widetilde P\|_{L^2}
\|\dive\delta w\|_{L^2}
\le
Cj\|\nabla\delta w\|_{L^2}
\le
\varepsilon\|\nabla\delta w\|_{L^2}^2
+
C_\varepsilon j^2.
\end{align}
Combining \eqref{eq:T1-final}, \eqref{eq:T2-estimate}, and
\eqref{eq:T4-estimate} proves \eqref{eq:G1-estimate}.
\end{proof}
Now we estimate $\mathcal{G}_2$.
\begin{lemma}\label{lemma:G2}
    With the previous notation we have, for every $\varepsilon>0$ and almost every $t \in (0,T)$
    \begin{align*}
        |\mathcal{G}_2| \lesssim \big|\int\rho_1
(D_1\widetilde\Pi-D_2\Pi_2)\cdot\delta w \big|
\le{}&
\varepsilon\|\nabla\delta w\|_{L^2}^2
+C
\|\sqrt{\rho_1}\delta w\|_{L^2}^2
+
C
\left(
1+\|\sqrt{\rho_2}\dot u_2\|_{L^2}^2
\right)x^2.
    \end{align*}
\end{lemma}
\begin{proof}
We introduce two shortcuts that will be used throughout the proof:
\[
Z:=\widetilde\Pi-\Pi_2,
\qquad
\mathcal Q:=\mathcal R\dive.
\]
Then $\|Z\|_{L^2}\lesssim x$ and $\delta u=\delta w+\nu^{-1} Z.$ Moreover, if $\mathbb P$ denotes the Leray projector, then $ \mathcal Q=I-\mathbb P,$ and hence $\mathcal Q$ is bounded on $L^p(\T)$ for every $1<p<\infty$. By the uniform bounds on the pressure and standard elliptic estimates on the torus,
\begin{equation}\label{eq:direct-potential-bounds}
|\bar P_2|
+\|\widetilde P\|_{L^\infty}
+\|\widetilde\Pi\|_{L^\infty}
+\|\nabla\widetilde\Pi\|_{L^4}
+\|\nabla\Pi_2\|_{L^4}
\lesssim1.
\end{equation}
Moreover, using \eqref{eq:direct-potential-bounds} together with \eqref{eq:elliptic-w}, we have
\begin{equation}\label{eq:direct-u2-bound}
\|\nabla u_2\|_{L^4}
\le
\|\nabla w_2\|_{L^4}
+\nu^{-1}\|\nabla\Pi_2\|_{L^4}
\lesssim
1+\|\sqrt{\rho_2}\dot u_2\|_{L^2}.
\end{equation}

\medskip
\noindent\textit{Step 1.}
For a general pressure law satisfying \eqref{pressure}, the transport equation for the pressure contains a source term. More precisely,
\begin{equation}\label{eq:P2eq}
\partial_tP_2+\dive(P_2u_2)
=-F_2,
\qquad
F_2:=
\bigl(\rho_2P'(\rho_2)-P_2\bigr)\dive u_2.
\end{equation}
Notice that $F_2$ vanishes for linear pressure laws. Similarly, the
intermediate pressure satisfies
\begin{equation}\label{eq:Ptildeeq}
\partial_t\widetilde P+\dive(\widetilde P u_1)
=-\widetilde F,
\qquad
\widetilde F(X_1)J_1=F_2(X_2)J_2.
\end{equation}
Applying $\mathcal R$ to \eqref{eq:P2eq} and \eqref{eq:Ptildeeq} gives
\begin{align*}
\partial_t\widetilde\Pi
=
-\mathcal Q(\widetilde P u_1)-\mathcal R\widetilde F,
\qquad
\partial_t\Pi_2
=
-\mathcal Q(P_2u_2)-\mathcal RF_2.
\end{align*}

Subtracting the two identities directly and using $\dive Z=\dive(\widetilde\Pi-\Pi_2)=\widetilde P-P_2$
(where we used $\overline{\widetilde P}=\overline{P_2}$), we obtain
\begin{align*}
\begin{aligned}
D_1\widetilde\Pi-D_2\Pi_2
&=
u_1\cdot\nabla\widetilde\Pi-u_2\cdot\nabla\Pi_2
-\mathcal Q(\widetilde P u_1-P_2u_2)
-\mathcal R(\widetilde F-F_2)
\\
&=
(\delta u\cdot\nabla)\widetilde\Pi
+(u_2\cdot\nabla)Z
-\mathcal Q\bigl(\widetilde P\,\delta u+(\dive Z)u_2\bigr)
-\mathcal R(\widetilde F-F_2)
\\
&=
(\delta u\cdot\nabla)\widetilde\Pi
-\mathcal Q(\widetilde P\,\delta u)
+
(u_2\cdot\nabla)Z
-\mathcal Q((\dive Z)u_2)
-\mathcal R(\widetilde F-F_2).
\end{aligned}
\end{align*}
Using now $ \delta u=\delta w+\nu^{-1} Z,$ we arrive at the decomposition
\begin{equation}\label{eq:direct-material-splitting}
\begin{aligned}
D_1\widetilde\Pi-D_2\Pi_2
=&
\underbrace{
(\delta w\cdot\nabla)\widetilde\Pi
-\mathcal Q(\widetilde P\,\delta w)
}_{I_1}+
\underbrace{
\frac1\nu
\Bigl(
(Z\cdot\nabla)\widetilde\Pi
-\mathcal Q(\widetilde P Z)
\Bigr)
}_{I_2}
\\
&+
\underbrace{
(u_2\cdot\nabla)Z
-\mathcal Q((\dive Z)u_2)
}_{I_3}
\underbrace{
-\mathcal R(\widetilde F_2-F_2)
}_{I_4}.
\end{aligned}
\end{equation}

\medskip
\noindent\textit{Step 2.} We now estimate the terms in \eqref{eq:direct-material-splitting} one by one. We shall use the following identity. If $Y$ is a gradient field, then
\begin{equation}\label{eq:direct-transport-identity}
\mathcal A(v,Y)
:=
(v\cdot\nabla)Y-\mathcal Q((\dive Y)v)
=
-(\nabla v)^T Y
+
\mathcal Q
\Bigl(
\bigl[
\nabla v+(\nabla v)^T-(\dive v)I
\bigr]Y
\Bigr).
\end{equation}
The proof of \eqref{eq:direct-transport-identity} is postponed to the end of the argument. Its main feature is that all derivatives on the right hand side fall on $v$. In particular, the $L^p$ boundedness of $\mathcal Q$ and H\"older's inequality immediately give
\begin{equation}\label{eq:direct-transport-estimate}
\|\mathcal A(v,Y)\|_{L^p}
\lesssim
\|\nabla v\|_{L^q}\|Y\|_{L^r},
\qquad
\frac1p=\frac1q+\frac1r,
\end{equation}
whenever $1<p<\infty$ and the exponents on the right-hand side are admissible.

We first estimate $I_1$. Since $\widetilde P
=
\dive\widetilde\Pi+\bar P_2, $ we may write $I_1
=
\mathcal A(\delta w,\widetilde\Pi)
-\bar P_2\,\mathcal Q\delta w.$ Applying \eqref{eq:direct-transport-estimate} with $p=q=2$ and $r=\infty$, and using \eqref{eq:direct-potential-bounds}, we obtain
\begin{align*}
\begin{split}
\|I_1\|_{L^2}
&\lesssim
\|\nabla\delta w\|_{L^2}
\|\widetilde\Pi\|_{L^\infty}
+
|\bar P_2|\,\|\mathcal Q\delta w\|_{L^2}\lesssim
\|\nabla\delta w\|_{L^2}.
\end{split}
\end{align*}
Indeed, $\mathcal Q\delta w$ has zero mean, since it is the gradient of a periodic function. Hence Poincar\'e's inequality, the commutation of $\mathcal Q$ with derivatives, and its $L^2$ boundedness give $\|\mathcal Q\delta w\|_{L^2}
\lesssim
\|\nabla\mathcal Q\delta w\|_{L^2}
\lesssim
\|\nabla\delta w\|_{L^2}. $

For $I_2$, by H\"older's inequality and the boundedness of $\mathcal Q$ on $L^{4/3}$,
\begin{align*}
\begin{split}
\|I_2\|_{L^{4/3}}
&\lesssim
\|Z\|_{L^2}\|\nabla\widetilde\Pi\|_{L^4}
+
\|\widetilde P\|_{L^\infty}\|Z\|_{L^{4/3}}
\lesssim
\|Z\|_{L^2}
\lesssim x.
\end{split}
\end{align*}
Here we used \eqref{eq:direct-potential-bounds} and the finite measure of the torus, which gives $\|Z\|_{L^{4/3}}\lesssim\|Z\|_{L^2}$.

For $I_3$ we first notice that  $Z=\widetilde\Pi-\Pi_2$ is itself a gradient field, and we may apply \eqref{eq:direct-transport-estimate} to $I_3$ with $v=u_2$, $Y=Z$, $p=4/3$, $q=4$, and $r=2$. Therefore,
\begin{align*}
\|I_3\|_{L^{4/3}}
&\lesssim
\|\nabla u_2\|_{L^4}\|Z\|_{L^2}
\lesssim
\left(
1+\|\sqrt{\rho_2}\dot u_2\|_{L^2}
\right)x,
\end{align*}
where we used \eqref{eq:direct-u2-bound} and $\|Z\|_{L^2}\lesssim x$.

\medskip 
$I_4$ requires a preliminary estimate. Indeed we first estimate $\widetilde F-F_2$ in a negative Sobolev norm in the spirit of \cref{lem:stability}. Let $\varphi$ be a smooth scalar function. Recalling that $\widetilde F(X_1)J_1=F_2(X_2)J_2$, a change of variables gives
\begin{align}\label{eq:F-difference}
\int(\widetilde F-F_2)\varphi
=
\int
F_2(X_2)J_2
\bigl(
\varphi(X_1)-\varphi(X_2)
\bigr) .
\end{align}
Since $F_2=\bigl(\rho_2P'(\rho_2)-P_2\bigr)\dive u_2$ and $|\rho P'(\rho)-P(\rho)|\lesssim \rho$ (see the proof of \cref{lem:stability}), using $\rho_2(X_2)J_2=\rho_0$ we obtain
\begin{align}\label{eq:F2-flow-bound}
|F_2(X_2)J_2|\lesssim \rho_0\,|\dive u_2(X_2)|.
\end{align}
Moreover, the Sobolev maximal-function inequality yields
\[
|\varphi(X_1)-\varphi(X_2)|
\lesssim
|X_1-X_2|
\left(
\mathcal M|\nabla\varphi|(X_1)
+
\mathcal M|\nabla\varphi|(X_2)
\right).
\]
Combining this with \eqref{eq:F-difference} and \eqref{eq:F2-flow-bound}, and applying H\"older's inequality, we obtain
\begin{align*}
\begin{split}
\left|
\int (\widetilde F-F_2)\varphi
\right|
\lesssim& \| \sqrt{\rho_0} \delta X\|_{L^2} \| \rho^{\frac{1}{4}}_0 \dive u_2(X_2) \|_{L^4}
\sum_{i=1}^2
\left(
\int
\rho_0
|\mathcal M\nabla\varphi(X_i)|^4
\right)^{1/4} \\ \lesssim& \; x \|  \dive u_2\|_{L^4}
\|\nabla\varphi\|_{L^4} \lesssim x
\left(
1+\|\sqrt{\rho_2}\dot u_2\|_{L^2}
\right)
\|\nabla\varphi\|_{L^4}
\end{split}
\end{align*}
where we have used \eqref{eq:direct-u2-bound}. We now pass to $I_4$. Let $\psi\in L^4(\T;\mathbb{R}^2)$ and $\mathcal{R}^* \psi=-\Delta^{-1} \dive \psi,$ by duality
\begin{align*}
\left| \int 
\mathcal R(\widetilde F-F_2)\cdot\psi \right|
=\left|
\int
(\widetilde F-F_2)\mathcal R^*\psi \right|  \lesssim x
\left(
1+\|\sqrt{\rho_2}\dot u_2\|_{L^2}
\right)
\| \nabla \mathcal R^*\psi\|_{L^4}
\end{align*}
Since $\nabla\mathcal R^*\psi=-\mathcal Q\psi$, taking the supremum over $\|\psi\|_{L^4}\le1$, we conclude that
\begin{align*}
\|I_4\|_{L^{4/3}}
\lesssim
x
\left(
1+\|\sqrt{\rho_2}\dot u_2\|_{L^2}
\right).
\end{align*}
Collecting the estimates for $I_1$, $I_2$, $I_3$ and $I_4$ we get
\begin{equation}\label{eq:direct-parts-summary}
\begin{aligned}
\|I_1\|_{L^2}
\lesssim
\|\nabla\delta w\|_{L^2}, \qquad 
\|I_2+I_3+I_4\|_{L^{4/3}}
\lesssim
\left(
1+\|\sqrt{\rho_2}\dot u_2\|_{L^2}
\right)x.
\end{aligned}
\end{equation}
\medskip

\noindent \textit{Step 3.} Combining \eqref{eq:Gdef} with \eqref{eq:direct-material-splitting} and \eqref{eq:direct-parts-summary} gives
\begin{align*}
|\mathcal{G}_2|
&\lesssim 
\bigg|\int\rho_1
\delta w \cdot I_1 \bigg|
+
\bigg|\int\rho_1
\delta w \cdot (I_2+I_3+I_4) \bigg|
\\
&\lesssim
\|\sqrt{\rho_1}\delta w\|_{L^2}
\bigl(\|\nabla\delta w\|_{L^2}+x\bigr)
+
\|\rho_1\delta w\|_{L^4}\,
x\left(1+\|\sqrt{\rho_2}\dot u_2\|_{L^2}\right).
\end{align*}
By the uniform bound on $\rho_1$, Sobolev embedding, and the weighted Poincar\'e inequality,
\[
\|\rho_1\delta w\|_{L^4}
\lesssim
\|\delta w\|_{H^1}
\lesssim 
\|\sqrt{\rho_1}\delta w\|_{L^2}
+
\|\nabla\delta w\|_{L^2}.
\]
Therefore, Young's inequality yields, for every $\varepsilon>0$,
\begin{align*}
|\mathcal{G}_2|
\leq{}&
\varepsilon\|\nabla\delta w\|_{L^2}^2
+
C\|\sqrt{\rho_1}\delta w\|_{L^2}^2
+
C
\left(
1+\|\sqrt{\rho_2}\dot u_2\|_{L^2}^2
\right)x^2.
\end{align*}
This concludes the proof up to proving \eqref{eq:direct-transport-identity}.

\medskip

\noindent \textit{Step 4.} To show  \eqref{eq:direct-transport-identity} we use that $\mathbb P \mathcal{Q} = 0$ and $ \mathcal{Q} \mathcal{Q} = \mathcal{Q} $ to write
\begin{align*}
\mathcal A(v,Y) & =\mathbb P [ (v\cdot\nabla)Y-\mathcal Q((\dive Y)v) ]+ \mathcal Q [ (v\cdot\nabla)Y-\mathcal Q((\dive Y)v) ]
\\ & =
\mathbb P\bigl((v\cdot\nabla)Y\bigr)
+
\mathcal Q\bigl((v\cdot\nabla)Y-(\dive Y)v\bigr).
\end{align*}
Now we study the two addends on the right hand side. For the first one, since $Y$ is a gradient field $\nabla Y$ is symmetric. This gives 
\[
\mathbb P((v\cdot\nabla)Y)
= \mathbb P(
\nabla(v\cdot Y)-(\nabla v)^T Y) =-  \mathbb P((\nabla v)^T Y).
\]
On the other hand, for the second addend we have 
\begin{align*}
\dive\bigl((v\cdot\nabla)Y-(\dive Y)v\bigr) & = \operatorname{tr}( \nabla v \nabla Y) -\dive Y \dive v
\\ 
\dive\bigl((\nabla v-(\dive v)I)Y\bigr) & =\operatorname{tr}( \nabla v \nabla Y) -\dive Y \dive v,
\end{align*}
in particular they coincide. Then, since $\mathcal Q=\nabla\Delta^{-1}\dive$ depends only on the divergence, we obtain
\[
\mathcal Q\bigl((v\cdot\nabla)Y-(\dive Y)v\bigr)
=
\mathcal Q\bigl((\nabla v-(\dive v)I)Y\bigr).
\]
Combining the previous identities proves \eqref{eq:direct-transport-identity}.
\end{proof}
\section{Proof of the main result}\label{sec proof of the main result}
Let $(\rho_1,u_1)$ and $(\rho_2,u_2)$ be two solutions of \eqref{main system}, in the sense of \cref{def:bounded-divergence-solution}, with the same initial data. Assume in addition that $(\rho_2,u_2)$ satisfies \eqref{eq:stronger-u2-regularity}. We fix an arbitrary time $T>0$. Throughout the proof, implicit constants may depend on the a priori bounds of the two solutions on $[0,T]$.

We set
\[
H^2(t):=\|\sqrt{\rho_1}\delta w(t)\|_{L^2}^2+j^2(t)
\]
and introduce the time-dependent functional
\[
F(t):=H^2(t)+\frac{x^2(t)}{t}\log\frac{e}{x^2(t)},\qquad 0<t\le T,
\]
with the displacement term defined to be zero when $x(t)=0$.

\underline{\textit{Step 1.}} We first establish the initial value of this functional. The Jacobian equation and the uniform Jacobian bounds imply
\[
\|J_i(t)-J_i(s)\|_{L^\infty}\le C_T\int_s^t\|\dive u_i(\tau)\|_{L^\infty}\,d\tau,\qquad 0\le s\le t\le T.
\]
Hence $j$ is continuous and $j(0)=0$, since $J_1(0)=J_2(0)=1$. Together with \eqref{eq:contdeltaw} and the common initial data, this gives that $H$ is continuous and $H(0)=0$. Moreover, the flow equations and the kinetic-energy bounds yield
\[
|x(t)-x(s)|\le\int_s^t\left(\|\sqrt{\rho_1}u_1(\tau)\|_{L^2}+\|\sqrt{\rho_2}u_2(\tau)\|_{L^2}\right)\,d\tau\le C_T(t-s).
\]
Thus $x$ is Lipschitz continuous, $x(0)=0$, and $x(t)\le C_Tt$. Since $r\mapsto r\log(e/r)$ is increasing near zero, it follows that, for sufficiently small $t>0$,
\[
\frac{x^2(t)}{t}\log\frac{e}{x^2(t)}\lesssim_T t\left(1+\log\frac1t\right)\longrightarrow0\qquad\text{as }t\downarrow0.
\]
Consequently, $F(t)\to0$ as $t\downarrow0$, and we define $F(0)=0$.

\underline{\textit{Step 2.}} By these limits, we may choose $0<T_0\le\min\{T,1\}$ such that
\[
x(t)\le e^{-2},\qquad 0\le F(t)\le e^{-2}\qquad\text{for every }t\in(0,T_0].
\]
In particular, $H(t)\le e^{-1}$, so that $x(t)+H(t)<1$ on this interval. All differential inequalities in the following local argument hold for almost every $t\in(0,T_0)$. Denoting by
\[
\Lambda:=1+\|\nabla w_2\|_{L^\infty},\qquad \Gamma:=\Lambda+\|G_1\|_{L^\infty}+\|\sqrt{\rho_2}\dot u_2\|_{L^2}^2,
\]
and combining the estimates of \cref{lem:flowStability,lemma:rel-energy,lemma:G2} with Young's inequality, choosing its parameters sufficiently small to absorb the gradient terms into the dissipation, we obtain
\begin{equation}\label{eq:uniqueness-ode-system}
\begin{aligned}
\frac{d}{dt}H^2
&\lesssim \Gamma H(H+x)+\left(\Gamma+\|\dot u_2\|_{H^1}^2\right)x^2\log\frac{e}{x^2}+\|\dot u_2\|_{H^1}Hx\left(\log\frac{e}{x^2}\right)^{3/2},\\
\frac{d}{dt}x^2
&\lesssim xH+x^2+\Lambda x^2\log\frac{e}{x^2}.
\end{aligned}
\end{equation}
Here we used $\|\sqrt{\rho_1}\delta w\|_{L^2}\le H$, $j\le H$, and $x^2\le x^2\log(e/x^2)$. By definition,
\begin{equation}\label{eq:uniqueness-F-bounds}
H^2\le F,\qquad x^2\log\frac{e}{x^2}\le tF,\qquad x^2\le tF.
\end{equation}
Differentiating $F$, we obtain
\[
\frac{dF}{dt}=\frac{dH^2}{dt}+\frac{\log(e/x^2)-1}{t}\frac{dx^2}{dt}-\frac{x^2}{t^2}\log\frac{e}{x^2}.
\]
Since the last term is nonpositive, we are left with
\begin{equation}\label{eq:uniqueness-F-derivative}
\frac{dF}{dt}\le\frac{dH^2}{dt}+\frac{\log(e/x^2)-1}{t}\frac{dx^2}{dt}.
\end{equation}
These computations are first performed where $x>0$.

\underline{\textit{Step 3.}} Before estimating \eqref{eq:uniqueness-F-derivative}, we notice that, for $0<x\le e^{-2}$,
\[
\log\frac{e}{x^2}=\log\frac{e}{x^2\log(e/x^2)}+\log\left(\log\frac{e}{x^2}\right)\le\log\frac{e}{x^2\log(e/x^2)}+\frac12\log\frac{e}{x^2},
\]
and hence $\log(e/x^2)\le2\log\bigl(e/[x^2\log(e/x^2)]\bigr)$. Since $s\mapsto s\log^2(e/s)$ is increasing on $(0,e^{-2}]$, and $0\le x^2\log(e/x^2)\le tF\le e^{-2}$, we conclude that
\begin{equation}\label{eq:uniqueness-log-estimate}
\begin{aligned}
x^2\left(\log\frac{e}{x^2}\right)^3
&\le4x^2\log\frac{e}{x^2}\left(\log\frac{e}{x^2\log(e/x^2)}\right)^2\\
&\le4tF\left(\log\frac{e}{tF}\right)^2
\lesssim tF\left(\log\frac{e}{F}+\log\frac{e}{t}\right)^2.
\end{aligned}
\end{equation}
All products in this estimate are interpreted by continuous extension at zero. 

For the first inequality in \eqref{eq:uniqueness-ode-system}, we use \eqref{eq:uniqueness-F-bounds} and \eqref{eq:uniqueness-log-estimate}:
\begin{equation}\label{eq:uniqueness-H-estimate}
\begin{aligned}
\frac{d}{dt}H^2
&\lesssim\Gamma H(H+x)+\left(\Gamma+\|\dot u_2\|_{H^1}^2\right)x^2\log\frac{e}{x^2}+\|\dot u_2\|_{H^1}Hx\left(\log\frac{e}{x^2}\right)^{3/2}\\
&\lesssim\Gamma\sqrt F\left(\sqrt F+\sqrt{tF}\right)+\left(\Gamma+\|\dot u_2\|_{H^1}^2\right)tF+\|\dot u_2\|_{H^1}\sqrt F\left[x^2\left(\log\frac{e}{x^2}\right)^3\right]^{1/2}\\
&\lesssim\Gamma F+t\|\dot u_2\|_{H^1}^2F+\|\dot u_2\|_{H^1}\sqrt t\,F\left(\log\frac{e}{F}+\log\frac{e}{t}\right)\\
&\lesssim F\log\frac{e}{F}\left[\Gamma+t\|\dot u_2\|_{H^1}^2+\|\dot u_2\|_{H^1}\sqrt t\left(1+\log\frac{e}{t}\right)\right],
\end{aligned}
\end{equation}
where we have used $t\le1$ and $\log(e/F)\ge1$.

Since $\log(e/x^2)-1\ge0$, we may multiply the upper bound for $(x^2)'$ in \eqref{eq:uniqueness-ode-system} by this coefficient. First, \eqref{eq:uniqueness-F-bounds} and \eqref{eq:uniqueness-log-estimate} give
\[
x^2\left(\log\frac{e}{x^2}\right)^2=\left[x^2\log\frac{e}{x^2}\right]^{1/2}\left[x^2\left(\log\frac{e}{x^2}\right)^3\right]^{1/2}\lesssim tF\left(\log\frac{e}{F}+\log\frac{e}{t}\right).
\]
Consequently,
\begin{equation}\label{eq:uniqueness-x-estimate}
\begin{aligned}
\frac{\log(e/x^2)-1}{t}\frac{d}{dt}x^2
&\lesssim\frac{\log(e/x^2)}{t}\left(xH+x^2+\Lambda x^2\log\frac{e}{x^2}\right)\\
&\lesssim\frac{\sqrt F}{t}\left[x^2\left(\log\frac{e}{x^2}\right)^2\right]^{1/2}+\frac{x^2}{t}\log\frac{e}{x^2}+\frac{\Lambda x^2}{t}\left(\log\frac{e}{x^2}\right)^2\\
&\lesssim F\sqrt{\frac{\log(e/F)+\log(e/t)}{t}}+F+\Lambda F\left(\log\frac{e}{F}+\log\frac{e}{t}\right)\\
&\lesssim F\log\frac{e}{F}\left[\frac{1+\sqrt{\log(e/t)}}{\sqrt t}+\Lambda\left(1+\log\frac{e}{t}\right)\right].
\end{aligned}
\end{equation}

Combining \eqref{eq:uniqueness-F-derivative}, \eqref{eq:uniqueness-H-estimate} and \eqref{eq:uniqueness-x-estimate}, we conclude that
\begin{equation}\label{eq:uniqueness-osgood}
\frac{dF}{dt}\lesssim F\log\frac{e}{F}\,\Theta,
\end{equation}
where
\[
\Theta:=\Gamma+t\|\dot u_2\|_{H^1}^2+\sqrt t\,\|\dot u_2\|_{H^1}\left(1+\log\frac{e}{t}\right)+\Lambda\left(1+\log\frac{e}{t}\right)+\frac{1+\sqrt{\log(e/t)}}{\sqrt t}.
\]

We now verify that $\Theta\in L^1(0,T_0)$. The assumptions and \cref{lem:poicnare,lem:elliptic-estimates} give
\[
\Gamma\in L^1(0,T),\qquad t\|\dot u_2\|_{H^1}^2\in L^1(0,T),\qquad\Lambda\in L^{6/5}(0,T).
\]
For the third and fourth terms, Young's inequality yields
\[
\sqrt t\,\|\dot u_2\|_{H^1}\left(1+\log\frac{e}{t}\right)+\Lambda\left(1+\log\frac{e}{t}\right)\lesssim t\|\dot u_2\|_{H^1}^2+\left(1+\log\frac{e}{t}\right)^2+\Lambda^{6/5}+\left(1+\log\frac{e}{t}\right)^6.
\]
Every term on the right-hand side belongs to $L^1(0,T_0)$. Finally,
\[
\int_0^{T_0}\frac{1+\sqrt{\log(e/t)}}{\sqrt t}\,dt<\infty,
\]
which proves the claimed integrability of $\Theta$. Osgood's lemma applied to \eqref{eq:uniqueness-osgood}, using $F(t)\to0$ as $t\downarrow0$, yields $F(t)=0$ for every $t\in[0,T_0]$. In particular, $x(t)=H(t)=0$ throughout this interval.

\underline{\textit{Step 4.}} We first recover the equality of the solutions on $(0,T_0)$. Since $x=0$, the estimate $\|\delta\rho\|_{H^{-1}}\lesssim x$ gives $\rho_1=\rho_2$ almost everywhere on $(0,T_0)\times\T$. Moreover, \eqref{eq:deltawDef} implies $\delta u=\delta w$. The identity $H=0$ also gives $j=0$ and $\|\sqrt{\rho_1}\delta w\|_{L^2}=0$. Using these identities in the estimates of \cref{lemma:rel-energy,lemma:G2}, and choosing their parameters sufficiently small to absorb the gradient terms, the relative energy inequality yields
\[
\int_0^{T_0}\|\nabla\delta w(t)\|_{L^2}^2\,dt=0.
\]
Thus $\delta w(t,\cdot)$ is spatially constant for almost every $t\in(0,T_0)$. Since the total mass is positive, the identity $\|\sqrt{\rho_1}\delta w\|_{L^2}=0$ forces this constant to vanish. Consequently,
\[
(\rho_1,u_1)=(\rho_2,u_2)\qquad\text{almost everywhere on }(0,T_0)\times\T.
\]

Finally, we can extend the argument to $[0,T]$ by a standard continuation argument. Since $T>0$ was arbitrary, this concludes the proof of \cref{thm:uniqueness}.

\appendix
\section{}
In this appendix, we state and give a proof of some auxiliary results used throughout the paper. The following lemma involves a weighted Poincaré inequality which is a particular case of \cite[Lemma A.2]{DM2023}.
\begin{lemma}\label{lem:poicnare}
With the notation of \eqref{eq:notation} the following estimates hold for almost every $t\in(0,T)$ and every $v\in H^1(\T)$ and $i=1,2$,
\begin{equation}\label{eq:weighted-Poincare}
\|v\|_{H^1}
\lesssim 
\|\nabla v\|_{L^2}
+\|\sqrt{\rho_i}\,v\|_{L^2},
\end{equation}
Moreover,
\begin{equation}\label{eq:acceleration-Poincare}
\|\dot u_2\|_{H^1}
\lesssim \|\nabla\dot u_2\|_{L^2}.
\end{equation}
\end{lemma}
\begin{proof}
Let $\bar v$ denote the spatial mean of $v$. Normalization of the total mass yields 
\begin{equation*} 
\bar v=\int_{\T}\rho_i v\,dx-\int_{\T}\rho_i(v-\bar v)\,dx
\end{equation*}
The usual Poincar\'e inequality therefore gives
\[
\|v\|_{H^1}
\le C\left(
\|\nabla v\|_{L^2}
+\left|\int_{\T}\rho_i v\,dx\right|
\right),
\]
and \eqref{eq:weighted-Poincare} follows by Cauchy--Schwarz.
Moreover, integrating the momentum equation in \eqref{main system} over the torus yields, for almost every time,
\[
\int_{\T}\rho_2\dot u_2\,dx
=
\int_{\T}
\left(\mu\Delta u_2+(\lambda+\mu)\nabla\dive u_2-\nabla P_2\right)\,dx
=0.
\]
Applying the previous bound to $v=\dot u_2$ proves
\eqref{eq:acceleration-Poincare}.
\end{proof}
In the next lemma we collect some properties of $w_2$ that will be used throughout the paper. Recall that
\[
w_2=u_2-\frac1\nu\Pi_2,
\qquad
\Pi_2=\nabla\Delta^{-1}(P_2-\bar P_2).
\]
All the properties below are direct consequences of the regularity estimates for $u_2$ established in \cite{DM2023}; we collect them here for later use.

\begin{lemma}\label{lem:elliptic-estimates}
With the notation of \eqref{eq:notation}, for every $2\le p<\infty$,
\begin{equation}\label{eq:elliptic-w}
\|\nabla w_2\|_{L^p}\lesssim \|\rho_2\dot u_2\|_{L^2},
\qquad
\|\nabla^2w_2\|_{L^p}\lesssim \|\rho_2\dot u_2\|_{L^p}.
\end{equation}
Moreover, there exists $r_*>1$ such that
\begin{equation}\label{eq:w2-time-Linfty}
\nabla w_2\in L^{r_*}_{\rm loc}([0,\infty);L^\infty(\T)).
\end{equation}
In particular, $u_2$ is log-Lipschitz in space and, for almost every $t>0$,
\begin{equation}\label{eq:u2-loglip}
|u_2(t,x)-u_2(t,y)|
\lesssim
\left(1+\|\nabla w_2(t)\|_{L^\infty}\right)
|x-y|\log\frac{e}{|x-y|}.
\end{equation}
\end{lemma}

\begin{proof}
We first recall that
\[
\mu\Delta w_2+(\lambda+\mu)\nabla\dive w_2=\rho_2\dot u_2.
\]
If we denote by $\mathcal L:=\mu\Delta+(\lambda+\mu)\nabla\dive$, then $\mathcal L$ is a second-order elliptic operator, since $\mu>0$ and $\nu=\lambda+2\mu>0$. Standard $L^p$ elliptic estimates therefore give, for every $2\le p<\infty$,\footnote{The elliptic estimates for the Lam\'e operator $\mathcal L$ follow from the conditions $\mu>0$ and $\nu>0$. Indeed, if $\mathcal L f=g$, and if $\mathbb P$ denotes the Leray projector and $\mathcal Q:=I-\mathbb P$, then
\[
\mathcal L=\mu\Delta\mathbb P+\nu\Delta\mathcal Q.
\]
Equivalently,
\[
\mu\Delta\mathbb P f=\mathbb P g,
\qquad
\nu\Delta\mathcal Q f=\mathcal Q g.
\]
Thus the solenoidal and gradient components satisfy elliptic equations with coefficients $\mu$ and $\nu$, respectively.}
\[
\|\nabla^2w_2\|_{L^p}\lesssim \|\rho_2\dot u_2\|_{L^p},
\]
which proves the second estimate in \eqref{eq:elliptic-w}. Taking $p=2$ in the previous estimate, and using the fact that $\nabla w_2$ has zero mean on the torus, Poincar\'e's inequality gives
\[
\|\nabla w_2\|_{H^1}
\lesssim
\|\nabla^2w_2\|_{L^2}
\lesssim
\|\rho_2\dot u_2\|_{L^2}.
\]
Hence, by the two-dimensional Sobolev embedding, for every $2\le p<\infty$,
\[
\|\nabla w_2\|_{L^p}
\lesssim\|\nabla w_2\|_{H^1}
\lesssim\|\rho_2\dot u_2\|_{L^2},
\]
which proves the first estimate in \eqref{eq:elliptic-w}.

We now prove \eqref{eq:w2-time-Linfty}. Since we only need one time exponent strictly larger than one, we fix the spatial exponent $p=4$. Using the elliptic estimate above, the uniform bound on $\rho_2$, the Sobolev embedding and \eqref{eq:acceleration-Poincare}, we obtain
\[
\|\nabla^2w_2\|_{L^4}
\lesssim
\|\rho_2\dot u_2\|_{L^4}
\lesssim
\|\dot u_2\|_{H^1}
\lesssim
\|\nabla\dot u_2\|_{L^2}.
\]
Together with the previous $H^1$ estimate, this gives
\[
\|\nabla w_2\|_{W^{1,2}}
\lesssim
\|\sqrt{\rho_2}\dot u_2\|_{L^2},
\qquad
\|\nabla w_2\|_{W^{1,4}}
\lesssim
\|\nabla\dot u_2\|_{L^2}.
\]
Interpolating between $W^{1,2}$ and $W^{1,4}$ yields
\[
\|\nabla w_2\|_{W^{1,8/3}}
\lesssim
\|\sqrt{\rho_2}\dot u_2\|_{L^2}^{1/2}
\|\nabla\dot u_2\|_{L^2}^{1/2}.
\]
Since $8/3>2$, the embedding $W^{1,8/3}(\T)\hookrightarrow L^\infty(\T)$ gives
\begin{equation}\label{eq:w2-pointwise-Linfty}
\|\nabla w_2(t)\|_{L^\infty}
\lesssim
t^{-1/4}
\|\sqrt{\rho_2}\dot u_2(t)\|_{L^2}^{1/2}
\left(\sqrt t\,\|\nabla\dot u_2(t)\|_{L^2}\right)^{1/2}.
\end{equation}
By \eqref{eq:stronger-u2-regularity}, the product
\[
\|\sqrt{\rho_2}\dot u_2\|_{L^2}^{1/2}
\left(\sqrt t\,\|\nabla\dot u_2\|_{L^2}\right)^{1/2}
\]
belongs to $L^2_{\rm loc}(0,\infty)$, while $t^{-1/4}\in L^3(0,T)$ for every $T>0$. H\"older's inequality therefore gives
\[
\nabla w_2\in L^{6/5}(0,T;L^\infty(\T)),
\]
so that \eqref{eq:w2-time-Linfty} holds, for instance, with $r_*=6/5$.

It remains to prove the log-Lipschitz estimate. Since $\Pi_2$ is a gradient field,
\[
\operatorname{curl} u_2=\operatorname{curl} w_2,
\qquad
\dive u_2=\dive w_2+\frac1\nu(P_2-\bar P_2).
\]
Since $P_2$ is uniformly bounded,
\[
\|\operatorname{curl} u_2\|_{L^\infty}
+
\|\dive u_2\|_{L^\infty}
\lesssim
1+\|\nabla w_2\|_{L^\infty}.
\]
Arguing as in \cite[Corollary 1.1]{DM2023}, the standard log-Lipschitz estimate gives
\[
|u_2(t,x)-u_2(t,y)|
\lesssim
\left(
\|\dive u_2(t)\|_{L^\infty}
+
\|\operatorname{curl} u_2(t)\|_{L^\infty}
\right)
|x-y|\log\frac{e}{|x-y|},
\]
and \eqref{eq:u2-loglip} follows.
\end{proof}

\begin{remark}\label{rem:w2-better-integrability}
The exponent $r_*=6/5$ is chosen only to keep the proof simple and is not optimal. By interpolating between $W^{1,2}$ and $W^{1,q}$ for arbitrary $q>2$, and choosing the interpolation parameter sufficiently small, the same argument gives
\[
\nabla w_2\in L^r_{\rm loc}([0,\infty);L^\infty(\T))
\qquad
\text{for every }1\le r<2.
\]
Consequently, the log-Lipschitz seminorm of $u_2$ belongs to the same spaces, consistently with the $L^{2-\varepsilon}_tL^\infty_x$ regularity obtained in \cite{DM2023}.
\end{remark}
\subsection*{Acknowledgements}
\noindent
Lorenzo Pescatore gratefully acknowledges the partial support by the Gruppo Nazionale per l’Analisi Matematica, la Probabilità e le loro Applicazioni (GNAMPA) of the Istituto Nazionale di Alta Matematica (INdAM). Alessandro Violini is supported by the Deutsche Forschungsgemeinschaft (DFG) via the project “Inhomogeneous and compressible fluids: statistical solutions and dissipative anomalies” within the SPP 2410 “Hyperbolic Balance Laws in Fluid Mechanics: Complexity, Scales, Randomness” (CoScaRa).
\\
The first author is grateful to Gianluca Crippa as part of this work was carried out during L. Pescatore's visit to the Department of Mathematics and Computer Science at the University of Basel.
\bibliographystyle{plain}
\bibliography{biblio}

\end{document}